\documentclass[graybox]{svmult}

\usepackage{type1cm}        % activate if the above 3 fonts are
\usepackage{makeidx}         % allows index generation
\usepackage{graphicx}        % standard LaTeX graphics tool
\usepackage{multicol}        % used for the two-column index
\usepackage[bottom]{footmisc}% places footnotes at page bottom

\usepackage{newtxtext}       % 
\usepackage[varvw]{newtxmath}       % selects Times Roman as basic font

\makeindex             % used for the subject index
\usepackage{bm}
\usepackage{amsmath}
\usepackage{amsthm}
\usepackage{appendix}
\usepackage{graphicx} % Required for inserting images
\usepackage{subcaption}
\usepackage{comment}
\usepackage[colorlinks=true, linkcolor=blue, citecolor=green, urlcolor=blue]{hyperref}
\usepackage{float}
\usepackage{placeins}
\usepackage{amsfonts}
\usepackage{xcolor}
\usepackage{csquotes}
\usepackage{bm}
\usepackage{algorithm}
\usepackage{algpseudocode}
\numberwithin{theorem}{section}
\DeclareMathOperator*{\argmin}{arg\,min} 
\usepackage{chngcntr} % For counter control
\usepackage{listings}
\usepackage{xcolor} % For coloring syntax

\begin{document}

\title*{Concentrated real-pole uniform-in-time approximation of the matrix exponential}
\author{Stefan G\"{u}ttel\orcidID{0000-0003-1494-4478} and\\ Shuai Shao\orcidID{0009-0001-3147-9355}}
\institute{Stefan G\"{u}ttel \at Department of Mathematics, The University of Manchester, United Kingdom, \email{stefan.guettel@manchester.ac.uk}
\and Shuai Shao \at Department of Mathematics, The University of Manchester, United Kingdom, \email{shuai.shao-2@manchester.ac.uk}}
\maketitle

\abstract{We propose an asymptotically optimal choice of shared concentrated real poles of a family of rational approximants of time-dependent exponential functions $\exp(-tz)$ for $z \geq 0$ and $t$ in a positive time interval~$T$. Our result extends a classical result  by J.-E.~Andersson [\emph{J.~Approx.~Theory}, 32(2):85--95, 1981] on the asymptotic best rational approximation of $\exp(-z)$ with real poles. Numerical experiments demonstrate the near-optimality of our choice for various time ranges and for both small and large approximation degrees. An application of the uniform-in-time rational approximation using our proposed concentrated real poles to a linear constant-coefficient initial-value problem is also discussed.}

\section{Introduction}
The numerical approximation of the action of a parametric matrix exponential function on a vector, $\exp(-tA)\bm{b}$, with a symmetric positive semidefinite matrix $A \in \mathbb{R}^{N\times N}$, a vector $\bm{b} \in \mathbb{R}^N$, for all $t \in T$ in a positive time interval $T$, is an important problem in science and engineering. For example, the exact solution of the fundamental dynamical system $\bm{f}'(t)+A\bm{f}(t) = \bm{0}$ can be written as $\bm{f}(t) = \exp(-tA)\bm{b}$, where $\bm{b} = \bm{f}(0)$ is the initial vector. Such systems arise when solving evolution problems in geophysics~\cite{{borner2015three}, {qiu2019block}, {borner2025fast}}, chemical reactions~\cite{{munsky2006finite},{jahnke2007solving}}, epidemiology~\cite{allen2008introduction}, network analysis~\cite{{estrada2012physics},{benzi2013ranking}}, and exponential integration~\cite{hochbruck2010exponential,guttel2010rational,guttel2013rational}. For a comprehensive introduction to matrix functions we refer to~\cite{higham2008functions}.

The explicit evaluation of a matrix function $\exp(A)$ is usually only possible for small to medium-sized matrices. For example, MATLAB's  $\texttt{expm}$ function implements a scaling and squaring approach to compute $\exp(A)$, which first computes a Pad\'{e} approximation of $\exp(A/2^s)$ and then performs $s$ matrix multiplications with  $s$ being carefully chosen~\cite{{higham2005scaling},{al2010new}}. This method is very reliable  but the  cost of dense matrix operations is $O(sN^3)$ flops, which becomes prohibitively expensive for large~$N$. Even more problematic is the fact that the algorithm has to be rerun from scratch for each $\exp(-tA)$ when the evaluation time point~$t$ changes, unless in special cases of, e.g., constant time steps. 
Therefore, alternative approaches that avoid forming the matrix exponential $\exp(-tA)$ entirely and do not strongly depend on $t$ are desirable.

A very effective approach is to construct a family of \textit{shared-pole} rational approximants 
\begin{equation}\label{partial_fraction_form}
    r_{t,n}(z) = \alpha_0(t)+\frac{\alpha_1(t)}{z-\sigma_1}+\cdots+\frac{\alpha_n(t)}{z-\sigma_n},
\end{equation}
with poles $\sigma_i$ and scalar residues $\alpha_i(t)$ such that 
\[
    r_{t,n}(A)\bm{b} \approx \exp(-tA)\bm{b}
\]
uniformly for all $t \in T$. We refer to $n$ as the degree of approximation. It is important that the poles $\sigma_i$ of $r_{t,n}$ do \emph{not} depend on the time parameter~$t$. Therefore, the computation of $r_{t,n}(A)\bm{b} $ involves solutions of only $n$ shifted linear systems $(A-\sigma_iI)\bm{x}_i = \bm{b}$, independent of the number of time points to evaluate for. Since these linear systems are completely decoupled, they can be solved in parallel using well-developed algorithms such as sparse Cholesky factorization \cite[Chaper~11]{golub2013matrix} or iterative methods~\cite{Saad2003}. 

The shared-pole restriction may be less relevant  if the number of evaluation time points is very small, allowing for  a lower degree $\widehat{n}$ for the same target tolerance $\texttt{tol}$\footnote{The tolerance $\texttt{tol}$ is set for the time-uniform error, which will be introduced later in~\eqref{timeuniform_err}. More precisely, $n$ and $\widehat{n}$ are the minimal degrees needed for the time-uniform error to fall below $\texttt{tol}$ with and without the shared-pole condition, respectively.} by finding a best approximant for each time point independently. If we allow for  poles to be placed anywhere in the complex plane, it is known that the set of optimal poles for $\exp(-z)$ lies on a  parabolic curve~\cite{trefethen2006talbot}. However, in many applications, the degree ratio $n/\widehat{n}$ is  much smaller than the number of discrete time points of interest, for various choices of $\texttt{tol}$ and $T$ in consideration; see, e.g., \cite[Section~3.3]{borner2025fast}. In such situations, the constraint on shared poles yields a substantially smaller total number of shifted linear systems to be solved for a fixed target tolerance. Our discussion hereafter will therefore focus on shared-pole approximants only.

Reliable algorithms for computing a family of rational approximants with shared poles $\sigma_i \in \mathbb{C}$ are readily available. These include  rational Krylov fitting (RKFIT)~\cite{berljafa2017rkfit} and the set-valued adaptive Antoulas--Anderson (svAAA) method~\cite{lietaert2022automatic}. The RKFIT method optimizes the shared poles iteratively, using orthogonal transformations on rational Krylov spaces in order to   approximately minimize a relative misfit given by the sum of squares of weighted error of each rational approximant. It then finds all the best approximants in a fixed shared-pole  rational Krylov subspace via least squares projection. The svAAA  method optimizes the shared interpolation nodes in the barycentric representation of rational approximants in a greedy fashion, with the poles being implicitly encoded in the barycentric form. The two algorithms are extensions of~\cite{berljafa2015generalized} and~\cite{nakatsukasa2018aaa}, respectively, from a single rational approximant to a family of rational approximants with shared poles. Since in both algorithms the shared poles are computed from a generalized eigenvalue problem involving a pair of real-valued matrices, the poles appear as conjugate pairs in the complex plane.

In this paper we focus on the case where $\sigma_i<0$ is restricted to be real. This  can lead to a significant reduction in the overall computational costs for matrix-valued problems as the resulting shifted linear systems are real (i.e., no complex arithmetic required) and symmetric positive definite.  Real-pole time-uniform approximations have been researched in the literature, for example in \cite{druskin2009solution}. In this work, the poles are placed in the negatively reflected spectral interval of $A$, relating the pole optimization problem to solving a classical Zolotarev problem on the real line. However, this approach requires knowledge of $A$'s spectral interval, and it deteriorates as the spectral interval increases. This approach hence does not typically lead to mesh-independent integrators with matrices arising from discretizations of differential operators. On the other hand, this approach leads to approximants which are good for the \emph{infinite} time interval $T = [0, +\infty)$, which may be useful in some circumstances (e.g., if the final time point of interest is not known in advance). 
In contrast, our paper proposes an approach to real-pole optimization that applies to  \textit{any} symmetric positive semidefinite matrix $A$ but requires a  \textit{finite} time interval~$T$.  

Before we proceed, let us briefly review some classical results on the rational approximation of the exponential function. For a fixed time point ($t=1$), Kaufman and Taylor~\cite{kaufman1978uniform} conjectured that the best uniform real-pole approximant to $\exp(-z)$ has all its poles \textit{concentrated} at a single point; that is, $\sigma_1 = \sigma_2=\cdots=\sigma_n = \sigma$ where $\sigma = -cn$ for some positive constant $c$, and this conjecture was later proved by Borwein~\cite{borwein1983rational}. Saff~\textit{et al.}~\cite{saff1975geometric} showed that for $c = 1$ the associated best uniform error decreases at most $O(2^{-n})$ and they suggested that $c = 1$ is also the optimal choice~\cite{saff1977some}. In a later work by Andersson~\cite{andersson1981approximation}, it was proved that the optimal choice is $c = 1/\sqrt{2}$, and the associated best uniform error decreases like $O((\sqrt{2}+1)^{-n})$.

In this paper, we will extend Andersson's result from~\cite{andersson1981approximation} to a time interval $T$, considering the case where the poles of all rational approximants are concentrated at a point $\sigma:=\sigma(n, T)$, which depends on the degree of approximation $n$ and the time interval $T$. Now, the partial fraction form \eqref{partial_fraction_form} must be changed to 
\begin{equation}\label{parial_fraction_same_poles}
    r_{t,n}(z) = \alpha_0(t)+\frac{\alpha_1(t)}{z-\sigma}+\frac{\alpha_2(t)}{(z-\sigma)^2}+\cdots+\frac{\alpha_n(t)}{(z-\sigma)^n}.
\end{equation}
Although the evaluation of $r_{t,n}(A)\bm{b}$ with $r_{t,n}$ in \eqref{parial_fraction_same_poles} involves solving dependent shifted linear systems which cannot be decoupled and computed in parallel, a benefit is that all these systems involve the same matrix $A-\sigma I$. Hence, only one Cholesky factorization needs to be computed and can be reused for all solves. 

In order to assess the quality of approximation, let us assume without loss of generality that $\Vert \bm{b}\Vert_2=1$, and note that for all $t \in T$,  
\begin{equation}\label{error_bound}
    \begin{split}
    \Vert r_{t,n}(A)\bm{b}-\exp(-tA)\bm{b}\Vert_2 &\leq \max_{z \in \Lambda(A)}\vert r_{t,n}(z)-\exp(-tz)\vert  \\ &\leq \max_{z \in [0, +\infty)}\vert r_{t,n}(z)-\exp(-tz)\vert \\ &=: \rho(n,t).
    \end{split}
\end{equation}
Here, $\Lambda(A)$ denotes the set of $A$'s eigenvalues. 
We aim to construct shared concentrated real-pole rational approximants $r_{t,n}$ such that $\rho(n,t)$ is uniformly small for all $t \in T.$ In other words, we consider the \textit{time-uniform error} $\rho(n, T)$  defined as
\begin{equation}\label{timeuniform_err}
    \rho(n, T): = \sup_{t \in T}\rho(n,t).
\end{equation}

This paper is organized as follows. In Section \ref{theories} we propose a method to determine an asymptotically optimal location of a shared concentrated real pole~$\sigma(n,T)$. % that the time-uniform error corresponding to the best possible family of rational approximants with shared concentrated poles $\sigma(n,T)$ is almost the minimum over that with shared concentrated poles at all possible locations on the negative real axis. 
We exploit a classical result on the asymptotically optimal concentrated real pole of a rational approximant for $\exp(-z)$ (i.e., a single time point $t=1$) by Andersson~\cite{andersson1981approximation}. In Section~\ref{numer_eval} we discuss efficient algorithms for the numerical evaluation of the proposed approximants. In Section~\ref{numerical_examples} we provide numerical experiments which confirm the near-optimality of our choice of poles, and we also compute matrix exponential functions acting on a vector arising from an initial-value problem. In Section~\ref{discussions} we conclude and discuss potential future work.

\section{Optimal shared concentrated real poles}\label{theories}
For any $q > 0$, let $\mathcal{R}^n_q$ be the set of degree $n$ real rational functions with poles concentrated at $-nq $,  
\[
    \mathcal{R}^n_q = \left\{\frac{p_n(z)}{(z+nq)^n}:p_n \in \mathcal{P}_n\right\}.
\]
Here and in the sequel, $\mathcal{P}_n$ denotes the class of polynomials of degree at most~$n$. Furthermore, for each time point $t$, let 
\[
    \rho^t_n(q) = \inf\left\{\sup_{z \geq 0}\vert r_{t,n}(z)-\exp(-tz) \vert: r_{t,n} \in \mathcal{R}^n_q\right\}.
\] 
The following theorem is one of the main results of~\cite{andersson1981approximation}.

\begin{theorem}\label{thm1}
For any $q > 0$, 
\[
    \displaystyle\lim_{n\rightarrow +\infty}\rho^1_n(q)^{1/n} = \tilde{H}(q),
\]
where 
\begin{equation}\label{expression_for_convergence_rate}
   \tilde{H}(q) = \exp \left(\log\Big|\frac{\tilde{q}-1}{\tilde{q}+1}\Big|+q \Re(\tilde{q}^2)\right),
\end{equation}
and $\tilde{q}$ is the root in $\Im(z) \geq 0$ of the equation
\begin{equation}\label{cubic_equ}
        q(z^3-z)+1 = 0
\end{equation}
that has the smallest positive real part. The minimum of $\tilde{H}$ is $\sqrt{2}-1$, which is attained at $1/\sqrt{2}.$
\end{theorem}

\begin{proof}
    We refer the reader to \cite{andersson1981approximation}.
\end{proof}
Theorem \ref{thm1} suggests that $-n/\sqrt{2}$ is an asymptotically optimal location for a  concentrated real pole when approximating $\exp(-z)$ on the non-negative real axis, and its associated uniform error decreases like  $(\sqrt{2}+1)^{-n}$ as $n$ gets large. 

We now extend the result of Theorem \ref{thm1} to rational approximants of a family of exponential functions $\exp(-tz)$ for $t \in T := [t_{\min}, t_{\max}]$ with $0<t_{\min}<t_{\max}$. Let 
\[
    \rho^T_n(q) = \sup_{t\in T}\rho^t_n(q)
\]
be the time-uniform error corresponding to a family of best approximants with shared poles concentrated at $-nq$, over the time interval $T$. We propose an asymptotically optimal choice of $q$ that minimizes $\rho^T_n$ for large $n$, which is derived from the above function $\tilde{H}$. Let us summarize the result in the following theorem.

\begin{theorem}\label{asymptoticbestpoles_manyf}
    For any $q > 0$, let $G(q) = \displaystyle\lim_{n\rightarrow+\infty}\left(\rho^T_n(q)\right)^{1/n}$ be the asymptotic rate of convergence associated with a family of best rational approximants with the concentrated real pole $-nq$, over the time interval $T: = [t_{\min}, t_{\max}]$ with $0 < t_{\min} < t_{\max}$.  Then 
    \begin{equation}\label{argmin}
       \argmin_{q>0} G(q) = \argmin_{q>0}\displaystyle\sup_{z \in [qt_{\min}, qt_{\max}]} \tilde{H}(z),
    \end{equation}
where $\tilde{H}$ is the function defined in \eqref{expression_for_convergence_rate}. 
\end{theorem}
\begin{proof}
For any $t \in T$, by using a change of variable $w = tz$, we have 
\begin{align*}
    \rho^t_n(q) &=  \inf\left\{\sup_{z \geq 0}\vert r_{t,n}(z)-\exp(-tz) \vert: r_{t,n} \in \mathcal{R}^n_q\right\} \\
   &= \inf\left\{\sup_{w \geq 0}\vert r_{1,n}(w)-\exp(-w) \vert: r_{1,n} \in \mathcal{R}^n_{qt}\right\} \\ &= \rho^1_n(qt).
\end{align*}
Therefore, we have
\[
    \rho^T_n(q) = \displaystyle\sup_{t \in T}\rho^t_n(q) = \displaystyle\sup_{t \in T} \rho^1_n(qt),
\]
and 
\begin{equation}\label{connection_gH}
        G(q) =  \displaystyle\sup_{t \in T}\left(\lim_{n \rightarrow +\infty} \rho^1_n(qt)\right)^{1/n} = \displaystyle\sup_{z \in [qt_{\min}, qt_{\max}]} \tilde{H}(z),
\end{equation}
where the second equality in \eqref{connection_gH} follows from Theorem \ref{thm1}. The proof is completed by taking $\argmin$ on the left and right-hand sides of \eqref{connection_gH}. 
\end{proof}

Theorem \ref{asymptoticbestpoles_manyf} allows us to determine an asymptotically optimal location of the shared concentrated poles by searching for an optimal interval with fixed endpoint ratio $\tau: = t_{\max}/t_{\min}$ such that the supremum of $\tilde{H}$ over that  interval is minimized. The following theorem on the monotonicity of $\tilde{H}$ guarantees that this optimization problem has a unique solution. 

\begin{theorem}\label{h_tilde}
    The function $\tilde{H}$ defined in Theorem~\ref{thm1} is monotonically decreasing when $q \in (0, 1/\sqrt{2})$ and increasing when $q \in [1/\sqrt{2}, +\infty)$. Moreover, $\tilde{H}$ has the following asymptotic behavior
    \[
        \lim_{q \rightarrow +0} \tilde{H}(q) = \lim_{q \rightarrow +\infty} \tilde{H}(q) = 1.
    \]
\end{theorem}

\begin{proof}
    By Cardano's formula, the discriminant of the cubic equation~\eqref{cubic_equ} can be written as $\Delta = 1/4q^2-1/27$ for all $q > 0$, which has the critical value $q^* = 3\sqrt{3}/2$.
    
    When $0<q<3\sqrt{3}/2$, we have $\Delta > 0$ and the three roots are given by 
    \begin{align*}
        z_1 &= a+b, \\
        z_2 &= a\bm{w}+b\bm{w}^2,\\
        z_3 &= b\bm{w} + a\bm{w}^2,
        \end{align*}
        where $a = (-1/2q +\sqrt{\Delta})^{1/3}$ and $b = (-1/2q -\sqrt{\Delta})^{1/3}$ satisfy $b < a < 0$, and $\bm{w} = e^{2\pi i/3}$. In this case, $z_1$ is a negative real number, $z_2$ and $z_3$ are complex conjugate numbers on the right-half plane with $\Im(z_2) > 0$ and $\Im(z_3) < 0$. Therefore, we have 
        \begin{align}
            \tilde{H}(q) &= \Big|\frac{z_2-1}{z_2+1}\Big|\exp(q \Re(z_2^2)) \nonumber \\
            &= \frac{\Big| (-\frac{1}{2}a-\frac{1}{2}b-1)+(\frac{\sqrt{3}}{2}a- \frac{\sqrt{3}}{2}b)i\Big|}{\Big| (-\frac{1}{2}a-\frac{1}{2}b+1)+(\frac{\sqrt{3}}{2}a- \frac{\sqrt{3}}{2}b)i\Big|} \exp(q\Re((a\bm{w}+b\bm{w}^2)^2) \nonumber \\
            &= \sqrt{\frac{(a+b)+1}{(a+b)-1}}\exp(-\frac{1}{2}q((a+b)^2-2)) \nonumber \\
            &:= \sqrt{\frac{m+1}{m-1}}\exp(-\frac{1}{2}q(m^2-2)) \label{interm_form}.
        \end{align}
The third equality follows from the relation $ab = 1/3$ and $\bm{w}^3=1$, and we used a change of variable $m: = a+b$ for the fourth equality. Moreover, by noticing that 
\begin{equation}\label{q_m}
    -1 = q(a^3+b^3) = q(a+b)((a+b)^2-1) = qm(m^2-1),
\end{equation}
we can now represent the expression~\eqref{interm_form} merely in terms of $m$, 
\[
    \tilde{H}(q) = \tilde{h}(m) = \sqrt{\frac{m+1}{m-1}}\exp\biggl(\frac{m^2-2}{2m(m^2-1)}\biggr).
\]
Taking the implicit derivative with respect to $q$ in equation~\eqref{q_m} yields: 
\[
    q(3m^2-1)m'(q)+m(m^2-1) = 0,
\]
and hence 
\[
    m'(q) = \frac{1}{q^2(3m^2-1)}.
\]
Since 
\[
    m^2 = (a+b)^2 \geq 4ab = 4/3,
\]
we conclude that $m'(q) > 0$ and $m(q)$ is an increasing function for $0<q<3\sqrt{3}/2$. As $m(q) = -\infty$ when $q \rightarrow 0^+$ and $m(q) = -2/\sqrt{3}$ when $q \rightarrow 3\sqrt{3}/2^-$, it suffices to analyze the monotonicity of $\tilde{h}$ on the domain $(-\infty, -2/\sqrt{3})$. Note that the derivative of $\tilde{h}$ is given by 
\[
    \tilde{h}'(m) = -\frac{1}{2}\Big(\frac{m+1}{m-1}\Big)^{-3/2}m^{-2}(m-1)^{-4}(m^2-2)(3m^2-1)\exp\Biggl(\frac{m^2-2}{2m(m^2-1)}\Biggr),
\]
which is negative when $m \in (-\infty, -\sqrt{2})$ and positive when $m \in (-\sqrt{2}, -2/\sqrt{3})$. Since $m(1/\sqrt{2}) = -\sqrt{2}$, we conclude that $\tilde{H}(q)$ is monotonically decreasing when $q \in (0, 1/\sqrt{2})$ and increasing when $q \in (1/\sqrt{2}, 3\sqrt{3}/2)$. Moreover, we can obtain the asymptotic behavior 
\[
    \lim_{q \rightarrow 0^+}\tilde{H}(q) = \lim_{m \rightarrow -\infty}\tilde{h}(m) = 1.
\]

When $q > 3\sqrt{3}/2$, we have $\Delta < 0$ and all three roots of equation~\eqref{cubic_equ} are real, which can be obtained as the intersection of the real-valued function $g_1(x) = x^3-x$ and the constant function $g_2(x)=-1/q$. Therefore, $\tilde{q}$ is the smallest positive real root of~\eqref{cubic_equ}. Since $g'_1(x) < 0$ for $0< x < 1/\sqrt{3}$ and $g'_1(x) > 0 $ for $1/\sqrt{3} < x < 1$, and $g_1(1/\sqrt{3}) < -1/q < 0 = g_1(0)=g_1(1)$, we have $\tilde{q} \in (0, 1/\sqrt{3})$. By representing $q$ in terms of $\tilde{q}$, $\tilde{H}$ can also be written as a function of $\tilde{q}$:

\begin{equation}\label{h_tild}
     \tilde{H}(q) = \phi(\tilde{q})=\Biggl(\frac{1-\tilde{q}}{1+\tilde{q}}\Biggr)\exp\Biggl(\frac{\tilde{q}}{1-\tilde{q}^2}\Biggr).
\end{equation}
Taking the derivative yields

\[
\phi'(\tilde{q}) = \frac{1-3\tilde{q}^2}{(\tilde{q}-1)(\tilde{q}+1)^3}\exp\Biggl(\frac{\tilde{q}}{1-\tilde{q}^2}\Biggr),
\]
which is negative for all $\tilde{q} \in (0, 1/\sqrt{3})$, and hence implying $\phi$ is a monotonically decreasing function with $\tilde{q}$. When $q$ increases, its associated $\tilde{q}$ decreases, which implies $\phi$ increases and hence $\tilde{H}$ increases. Moreover, we can obtain another asymptotic behavior 
\[
\lim_{q \rightarrow +\infty}\tilde{H}(q) = \lim_{\tilde{q}\rightarrow 0^+}\phi(\tilde{q})=1.
\]
The continuity of $\tilde{H}$ at the critical value $3\sqrt{3}/2$ can be ensured by noticing 
\[
\lim_{q \rightarrow 3\sqrt{3}/2^-}\tilde{H}(q)=\lim_{m\rightarrow -2/\sqrt{3}^-}\tilde{h}(m)=(2-\sqrt{3})\exp(\sqrt{3}/2),
\] and
\[
\lim_{q \rightarrow 3\sqrt{3}/2^+}\tilde{H}(q)=\lim_{\tilde{q}\rightarrow 1/\sqrt{3}^-}\phi(\tilde{q})=(2-\sqrt{3})\exp(\sqrt{3}/2).
\]
This completes the proof. 
\end{proof}

Figure~\ref{Hhatfunction} is a numerical plot of $\tilde{H}(q)$ using MATLAB for a million logarithmically spaced points $q$ in $[10^{-6}, 10^6]$, with the minimal point $(1/\sqrt{2},\sqrt{2}-1)$ marked by a red asterisk. The plot confirms our results in Theorem~\ref{h_tilde}. Based on the simple structure of $\tilde{H}$, we can numerically determine the right-hand side of \eqref{argmin} as the unique $\widehat{q}>0$ such that $\tilde{H}(\widehat{q}t_{\min}) = \tilde{H}(\widehat{q}t_{\max})$. Any other choice corresponds to a shifted interval with the same endpoint ratio $\tau$, which would lead to a larger supremum. 

Table~\ref{table1} collects values of $\widehat{q}$ and $G(\widehat{q})$ for $5$ different time ranges $T = [t_{\min}, t_{\max}] = 1, [10^{-1}, 1], [10^{-2}, 1], [10^{-3}, 1], [10^{-4}, 1]$, respectively.  Correspondingly, the asymptotically optimal locations of the shared concentrated pole for rational approximants of degree $n$ in these time ranges are around $\sigma(n, T_0) = -0.71n$, $\sigma(n,T_1) = -1.70n$, $\sigma(n, T_2) = -2.67n$, $\sigma(n, T_3) = -4.31n$, $\sigma(n, T_4) = -7.47n$, respectively. 

While Table~\ref{table1} records optimal pole parameters for the case $t_{\max}=1$, a simple rescaling of time as used in Theorem~\ref{asymptoticbestpoles_manyf} can be used to get the optimal pole for any time interval $[t_{\min}, t_{\max}]$. The convergence rate only depends on the ratio
$\tau= t_{\max}/t_{\min}$.

\begin{figure}[ht]
\centering
        \includegraphics[width=0.7\textwidth]{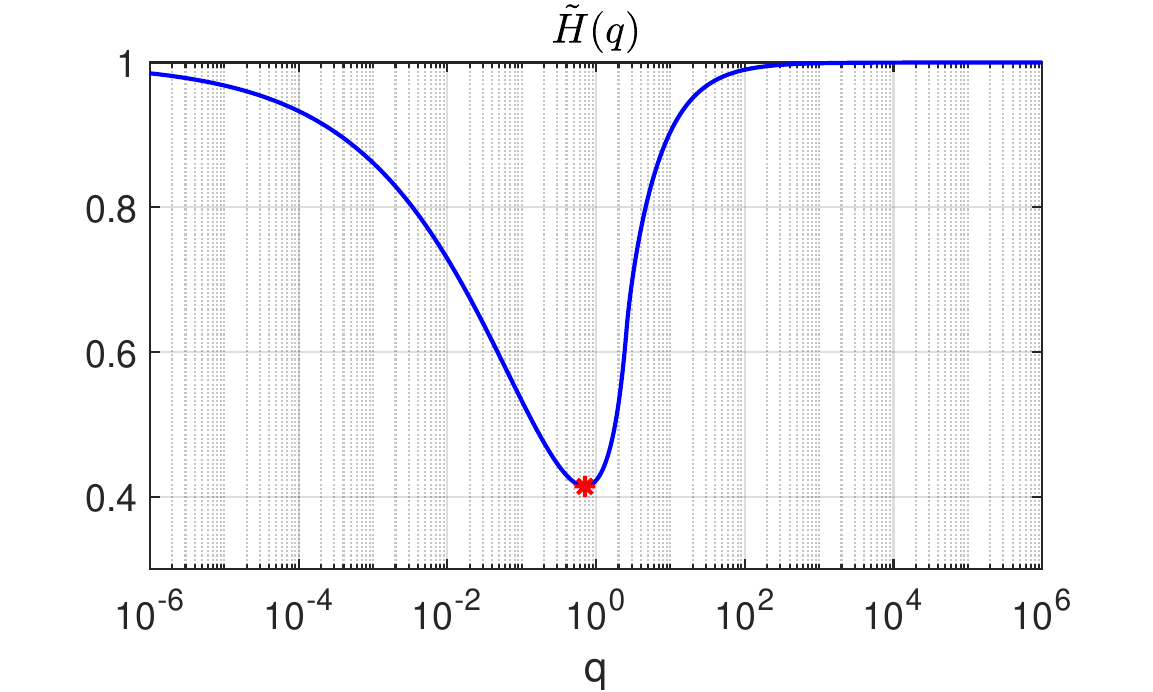}
        \caption{\textit{The plot of $\tilde{H}$ as a function of $q$. 
        }}
        \label{Hhatfunction}
\end{figure}

\begin{table}[h]
\centering
\caption{Asymptotically optimal parameters for approximating $\exp(-tz)$ uniformly for for $z\geq 0$ and over a time interval~$T =  [t_{\min}, t_{\max}]$. The optimal pole is at  $-n\widehat{q}$ and the expected convergence rate is $G(\widehat{q})$.}

\label{table1}

\setlength{\tabcolsep}{10pt}{
\begin{tabular}{c c c c c} 
                    
  \textbf{$T$} & \textbf{$t_{\min}$} & \textbf{$t_{\max}$} & \textbf{$\widehat{q}$} & \textbf{$G(\widehat{q})$} \\
  \hline
  $T_0$ & $1$  & $1$ & $0.71$ & $0.41$ \\
  
  $T_1$ & $10^{-1}$ & $1$ & $1.70$ & $0.49$ \\
  
  $T_2$ & $10^{-2}$ & $1$ & $2.67$ & $0.65$ \\
  
  $T_3$ & $10^{-3}$ & $1$ & $4.31$ & $0.79$ \\
  
  $T_4$ & $10^{-4}$ & $1$ & $7.47$ & $0.87$ \\

\end{tabular}
}
\end{table}

\section{Numerical evaluation of the approximants}\label{numer_eval}

Before verifying the near-optimality of our proposed concentrated real poles, we  briefly discuss several numerical methods for efficiently computing $\{r_{t,n}(A)\bm{b}\}_{t \in T}$, the action of a family of rational matrix functions on a vector, where $r_{t, n}$ is of the form \eqref{parial_fraction_same_poles} with the choice of $\sigma$ as suggested in Section~\ref{theories}.

\subsection{Rational Chebyshev interpolation}
The first method is to construct $r_{t, n}(z)$ as the rational Chebyshev interpolant of $\exp(-tz)$ for $z \in [0, +\infty)$. By considering the M\"{o}bius transformation $\widehat z = (z-\sigma)/(z+\sigma)$, this is equivalent to finding the polynomial Chebyshev interpolant $p_{t, n}(\widehat z)$ to the function $s_{t,\sigma}(\widehat z) = \exp(-t\sigma(1+\widehat z)/(1-\widehat z))$ for $\widehat z \in [-1, 1]$. If we represent $p_{t, n}(\widehat z)$ as
\[
p_{t, n}(\widehat z)=\displaystyle\sum_{k=0}^{n}c_{t,k} T_k(\widehat z),
\]
where $T_k(\widehat z) =  \cos(k \arccos(\widehat z))$ is the degree-$k$ Chebyshev polynomial of the first kind, then the coefficients $\{c_{t,k}\}$ can be computed from samples of $s_{t,\sigma}$ at Chebyshev points using fast Fourier Transform (FFT)~\cite[Chaper~4]{trefethen2019approximation}. Now we have 
\begin{equation}\label{cheb_interp}
r_{t, n}(A)\bm{b} = p_{t, n}(\widehat{A})\bm{b} = \displaystyle\sum_{k=0}^{n}c_{t,k}T_k(\widehat{A})\bm{b},
\end{equation}
where 
\[
\widehat{A} = (A-\sigma I)(A+\sigma I)^{-1}.
\]
The vector-valued Chebyshev approximants $T_k(\widehat{A})\bm{b}$ in \eqref{cheb_interp} can be computed efficiently by  running the three-term recurrence
\[
T_{k+1}(\widehat{A})\bm{b} = 2\widehat{A}T_{k}(\widehat{A})\bm{b}-T_{k-1}(\widehat{A})\bm{b}, \quad T_0(\widehat{A})\bm{b} = \bm{b}, \quad T_1(\widehat{A})\bm{b} = \widehat{A}\bm{b}. 
\]
Note that these $n+1$ vectors need to be generated only once and could be summed with different coefficients $\{c_{t,k}\}$ if the evaluation at various time points~$t$ is required. Overall, this amounts to solving $n$ shifted linear systems independent of the number of time points $t$ to be evaluated for. Furthermore, this Chebyshev interpolation approach guarantees near-best approximation quality, since the supremum approximation error of the Chebyshev interpolant is at most a factor $2+2\log(n+1)/\pi$ larger than that of the best approximation~\cite[Theorem~16.1]{trefethen2019approximation}. 

\subsection{Shift-and-invert Arnoldi method}

As an alternative we consider a shift-and-invert Arnoldi approach as in \cite{van2006preconditioning} with our optimized poles. Let  a symmetric positive semidefinite matrix~$A$, a vector $\bm{b}$, the pole $\sigma<0$ and an integer~$n$ be given. We  generate a Krylov basis matrix $V_{n}\in\mathbb{R}^{N\times n}$ with orthonormal columns that satisfies the Arnoldi relation
\[
(A - \sigma I)^{-1} V_{n} = V_{n} H_{n} + h_{n+1,n}  \bm{v}_{n+1}\bm{e}_n^T,
\]
where $H_n\in\mathbb{R}^{n\times n}$ is symmetric tridiagonal,  $\bm{v}_{n+1}\perp \mathrm{colspan}(V_n)$, and $\bm{e}_j$ denotes the $j$-th canonical unit vector in $\mathbb{R}^n$. The computation of this decomposition requires the solution of $n$ shifted linear systems. The shift-and-invert Arnoldi approximation to $\exp(-t A)\bm{b}$ is then given by
\[
\bm{f}_{t,n} = V_n \exp(-t [H_n^{-1} + \sigma I]) (\|\bm{b}\|_2 \bm{e}_1).
\]
This formula involves the evaluation  of an $n\times n$ matrix function, typically much smaller than the original problem size $N \gg n$. 

As opposed to the rational Chebyshev interpolant, the rational function $r_{t,n}$ underlying the shift-and-invert Arnoldi approximant $\bm{f}_{t,n} = r_{t,n}(A)\bm b$ has the potential to adapt to the discrete spectrum of $A$, and to the spectral components in $\bm b$. This property is known as \emph{spectral adaptivity} and it can lead to significant improvements in approximation accuracy; see, e.g., \cite{BG12}.

\section{Numerical experiments}\label{numerical_examples}
This section is dedicated to numerical experiments.  In Section~\ref{near-opt_check} we demonstrate that our proposed choice of poles derived using the approach in Section~\ref{theories} is almost indistinguishable from using best approximants produced by the Remez algorithm.  An application to a linear constant-coefficient initial-value problem is presented in Section~\ref{appl}.

\subsection{Verifying near-optimality}\label{near-opt_check}
Let us provide numerical evidence for the near-optimality of our choice of poles. Figure~\ref{n20} illustrates the time-uniform error $\rho^{T}_{n}(q)$ for various values of $q$, under degree $n = 20$, $T = T_1, T_2, T_3, T_4$, respectively. Here, we discretize each $T$ into $41$ logarithmically spaced time points, and use the scaled reciprocal $\gamma:= (nq)^{-1} = -\sigma^{-1}$ as the horizontal axis for a clearer visualization of our plots. The numerical method we use to compute $\rho^{t}_{n}(q)$ for all $t \in T$ is to solve an associated best polynomial approximation problem on $[-1,1]$ via the Chebfun $\texttt{'minimax'}$ command~\cite{filip2018rational} by exploiting  the following equivalence transformation \eqref{polynomial_approximation} using a change of variable $\widehat z  = (\gamma z+1)/(\gamma z-1)$,
\begin{equation}\label{polynomial_approximation}
    \begin{split}
    \rho^{t}_{n}(q) & = \inf\Big\{\sup_{z \in [0, +\infty)}\vert r_{t,n}(z)-\exp(-tz) \vert: r_{t,n} \in \mathcal{R}^n_q\Big\} \\
    & = \inf\Big\{\sup_{\widehat z \in [-1,1]}\vert p_{n}(\widehat z)-\exp(t\gamma^{-1}(1+\widehat z)/(1 - \widehat z)) \vert: p_{n} \in \mathcal{P}_n\Big\}.
    \end{split}
\end{equation}
(While the function $\exp(t\gamma^{-1}(1+\widehat z)/(1 - \widehat z))$ is not defined at $\widehat z = 1$, this is a removable singularity when setting the function value to be zero at that point.) Hence finding the optimal $r_{t,n}^* \in \mathcal{R}^n_q$ is equivalent to a best polynomial approximation problem. 
We can use this relation to numerically search for the optimal $q^*>0$ that minimizes $\rho^{t}_{n}(q)$.

In Figure~\ref{n20}, the black circles mark the corresponding time-uniform error $\rho^{T}_{20}(\widehat{q})$ computed using our choice of concentrated poles $\sigma(20, T_1)$, $\sigma(20, T_2)$, $\sigma(20,T_3)$, $\sigma(20,T_4)$, and they are visually close to the time-uniform errors $\rho^{T}_{20}(q^*)$, associated with the truly optimal computed concentrated pole parameter $q^*$, which are marked by red circles.

 \begin{figure}[ht]
    \centering
    \begin{subfigure}[b]{0.48\textwidth}
        \hspace*{-3mm}\includegraphics[width=1.1\textwidth]{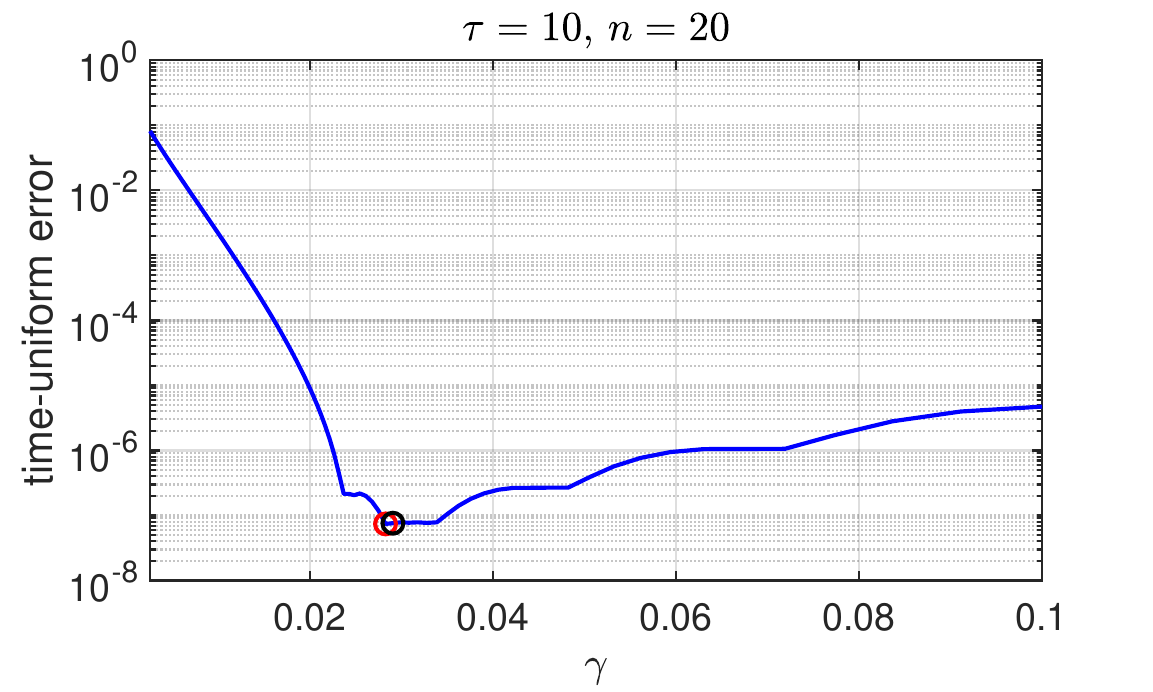}
        %\caption{$n = -3$}
        %\label{fig:sub1}
    \end{subfigure}
    \begin{subfigure}[b]{0.48\textwidth}
        \includegraphics[width=1.1\textwidth]{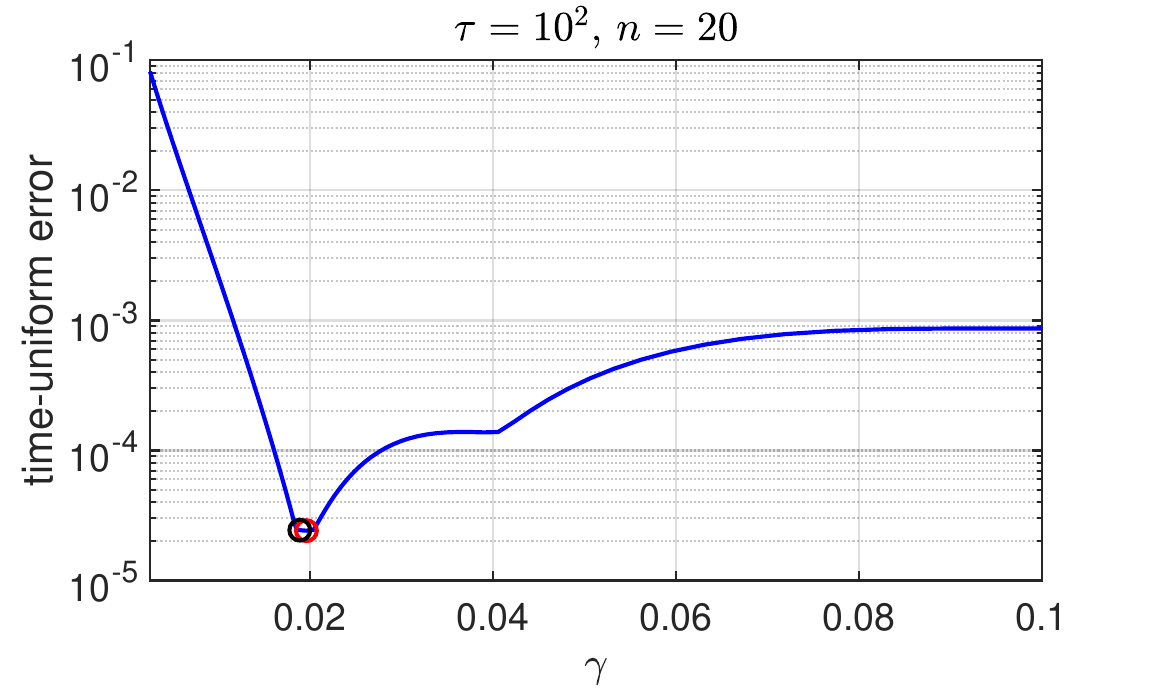}
        %\caption{$n = -2$}
        %\label{fig:sub2}
    \end{subfigure}
    
    \vspace{0.5cm} % Add some vertical space between rows
    
    \begin{subfigure}[b]{0.48\textwidth}
         \hspace*{-3mm}\includegraphics[width=1.1\textwidth]{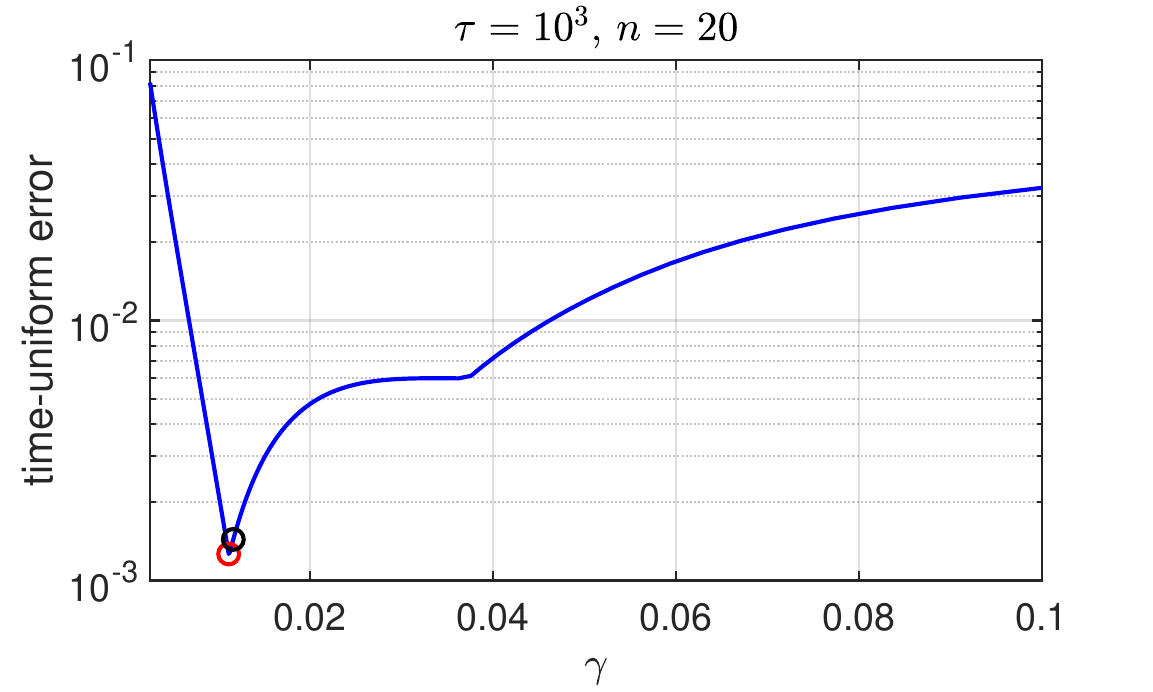}
        %\caption{$n = -1$}
        %\label{fig:sub3}
    \end{subfigure}
    \begin{subfigure}[b]{0.48\textwidth}
        \includegraphics[width=1.1\textwidth]{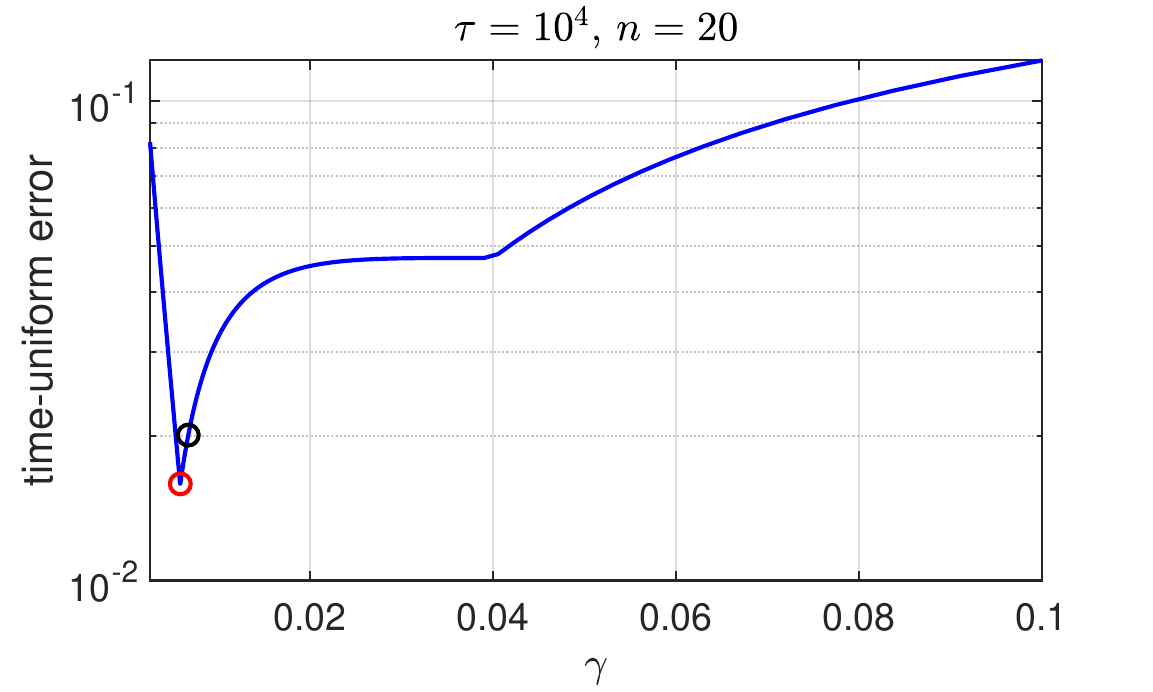}
        %\caption{$n = 0$}
        %\label{fig:sub4}
    \end{subfigure}
    
    \caption{\textit{Time-uniform error when using our choice of concentrated poles (marked by black circles) and numerically computed optimal choice of poles (marked by red circles), for different time ratios $\tau = t_{\max}/t_{\min}$. The degree is $n = 20$ in all cases.} }
    \label{n20}
\end{figure}

Table \ref{table2} collects the values of $\rho^{T}_{20}(\widehat{q})$ and $\rho^{T}_{20}(q^*)$ for $T = T_1, T_2, T_3,T_4$, and one can see that our choices of concentrated poles are near-best in all these cases. 

\begin{table}[ht]
\centering
\caption{Time-uniform approximation errors \textit{$\rho^{T}_{20}(\widehat{q})$ and \textbf{$\rho^{T}_{20}(q^*)$}, achieved with our asymptotically optimal poles versus the true optimal poles, respectively,  for different time intervals  $T_1$, $T_2$, $T_3$, $T_4$.}}
\label{table2}

\setlength{\tabcolsep}{10pt}\begin{tabular}{c c c}
 
  \textbf{$T$} & \textbf{$\rho^{T}_{20}(\widehat{q})$} & \textbf{$\rho^{T}_{20}(q^*)$} \\
  \hline
  $T_1$ & $7.75\times 10^{-8}$ & $7.15\times 10^{-8}$ \\
  
  $T_2$  & $2.44 \times 10^{-5}$ & $2.41 \times 10^{-5}$ \\
  
  $T_3$  & $1.40 \times 10^{-3}$ & $1.30 \times 10^{-3}$ \\
  
  $T_4$  & $2.02 \times 10^{-2}$ & $1.59\times 10^{-2}$ \\

\end{tabular}

\end{table}

Figure \ref{full_plot} shows the time-uniform error using our choices of poles, $\rho^{T}_{n}(\widehat{q})$, and the computed truly optimal poles, $\rho^{T}_{n}(q^*)$, for degrees~$n$ ranging from $1$ to $40$, and time ranges $T = T_0, T_1, T_2, T_3, T_4$. It can be seen from the figure that our choice of poles is near-optimal for all degrees and time ranges in consideration. The time-uniform error under the single time point $T_0$ stagnates for degree $n \geq 36$ as the level of machine precision is reached.

\begin{figure}[H]
    \centering
    \includegraphics[width = 0.8\linewidth]{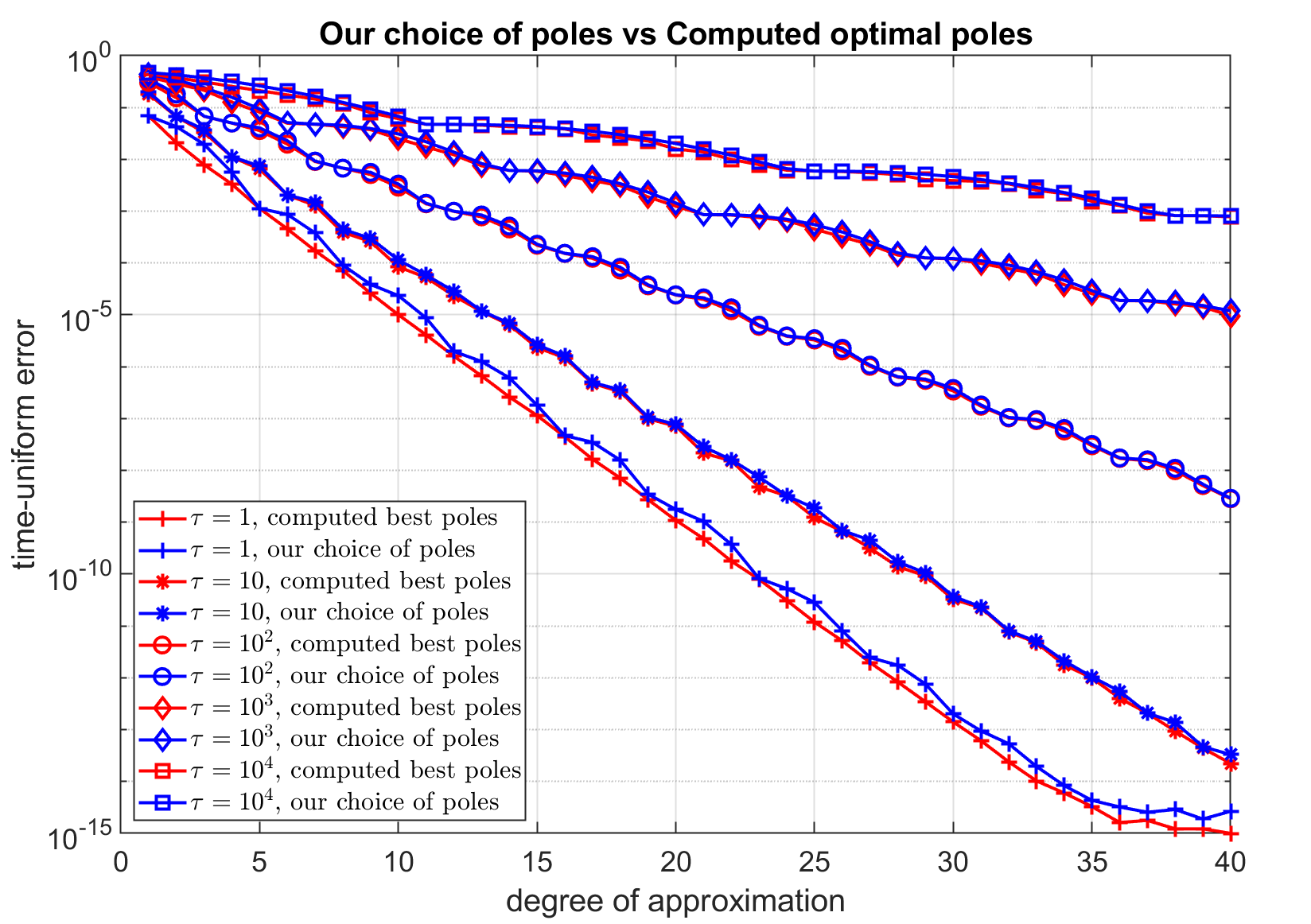}
    \caption{\textit{Time-uniform error using our choice of poles and the numerically computed optimal poles for the ranges $T_0$, $T_1$, $T_2$, $T_3$, $T_4$, for degree $n $ from $1$ to $40$.}}
    \label{full_plot}
\end{figure}

\subsection{An application to matrix exponential function}\label{appl}
We are now concerned with the approximation of the action of matrix exponential functions on a vector using our time-uniform approximants. Let us consider the linear constant-coefficient initial value problem
\begin{equation}\label{matexp}
    \bm{u}'(t)+L\bm{u}(t) = \bm{0}, \hspace{.5cm}\bm{u}(0) = \bm{u}_0.
\end{equation}
Here, $L \in \mathbb{R}^{4761\times 4761}$ is a sparse symmetric matrix arising from a finite-difference discretization of the scaled 2D Laplace operator $-0.2\Delta$ on the domain $\Omega = [-1,1]^2$ with  homogeneous Dirichlet boundary conditions, and $\bm{u}_0$ corresponds to the discretization of the initial function $u_0(x,y) = (1-x^2)(1-y^2)e^x$ on $\Omega$. 
We are interested in obtaining the solution $\bm{u}(t)$ of ~\eqref{matexp} at several time points $T = \{t_j\}_{j=1}^{s}$. A similar example has been considered in \cite{trefethen2006talbot} and~\cite[Section~6.2]{berljafa2017rkfit} (see also \cite{berljafa2015pole} for the corresponding MATLAB code).

Figure~\ref{100_2by2} shows the 2-norm error of the rational approximation $\widehat{\bm{u}}(t) := r_{t,n}(L)\bm{u}_0$ to the exact solution $\bm{u}(t) := \exp(-tL)\bm{u}_0$ for $41$ log-spaced time points in $[10^{-3}, 1]$ using the form~\eqref{parial_fraction_same_poles}, and concentrated poles $\sigma$ selected according to Theorem~\ref{asymptoticbestpoles_manyf}, along with its associated error bound $\rho(n,t)$ in~\eqref{error_bound}, for degrees  $n = 8, 14, 20, 26$. The approaches we use to compute $\widehat{\bm{u}}(t)$ are the rational Chebyshev interpolation and the shift-and-invert Arnoldi method, which are described in Section~\ref{numer_eval}. We also compute an ``approximate error bound'' by running the rational Chebyshev method for a diagonal matrix having ``dense'' spectrum on $[0,+\infty)$ and with a vector of all ones. This allows us to estimate the scalar approximation error of the rational interpolants. We find that the error bounds at the two time endpoints attain almost the same level, which is the maximum throughout the time range. This aligns with our choice of concentrated poles, which is constructed by minimizing the right-hand side of the equation~\eqref{argmin}, balancing the approximation over the whole time interval. We also see that the shift-and-invert Arnoldi method delivers significantly better approximants than the bound suggests. This is due to the spectral adaptation of the underlying rational functions. These effects would reduce if we chose a finer discretization of the differential operator or a rougher right-hand side vector (e.g., a random $\bm{u}_0$). 

%The method we use to compute $\widehat{\bm{u}}(t)$ is to first find the best scalar rational approximation that minimizes $\rho(n,t)$ in~\eqref{error_bound} given the set of concentrated poles using RKFIT~\cite{berljafa2017rkfit} with $0$ iterations and then compute the associated rational operator acting on $\bm{u}_0$. Numerical experiments show that similar plots can be created for other degrees and time ranges as well. 

\begin{figure}[h]
    \centering
    \begin{subfigure}[b]{0.48\textwidth}
        \includegraphics[width=1.1\textwidth]{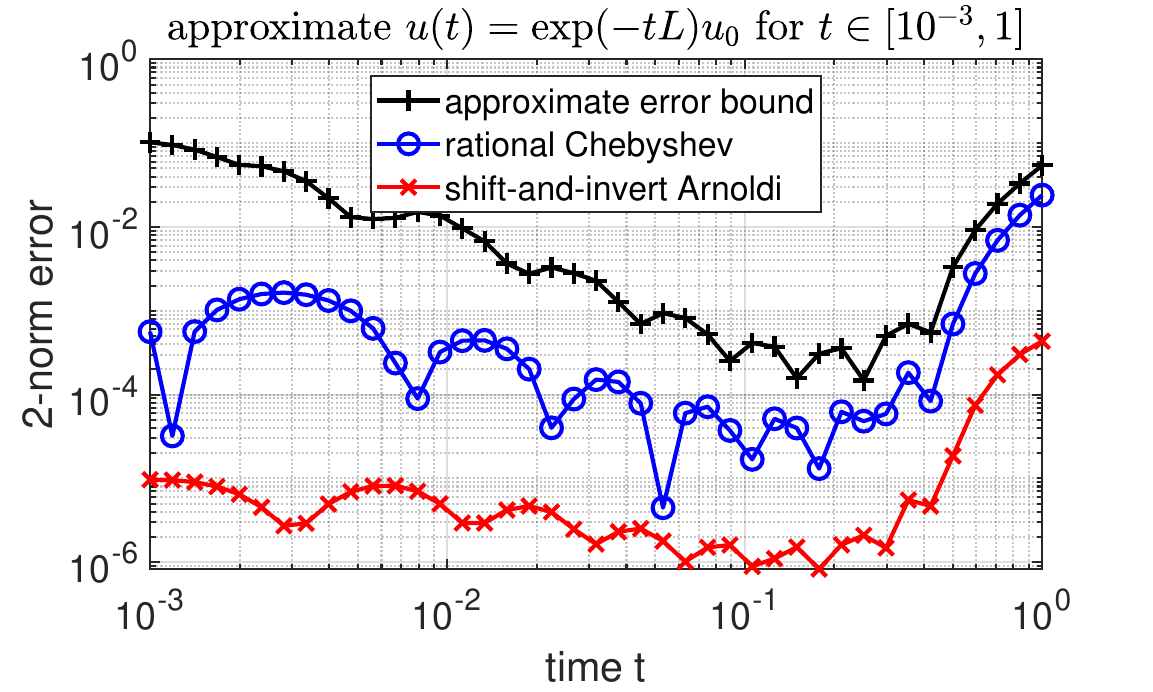}
        \caption{\textit{$n = 8$, $\tau = 10^3$.}}
        
        %\label{fig:sub1}
    \end{subfigure}
    \hfill
    \begin{subfigure}[b]{0.48\textwidth}
        \includegraphics[width=1.1\textwidth]{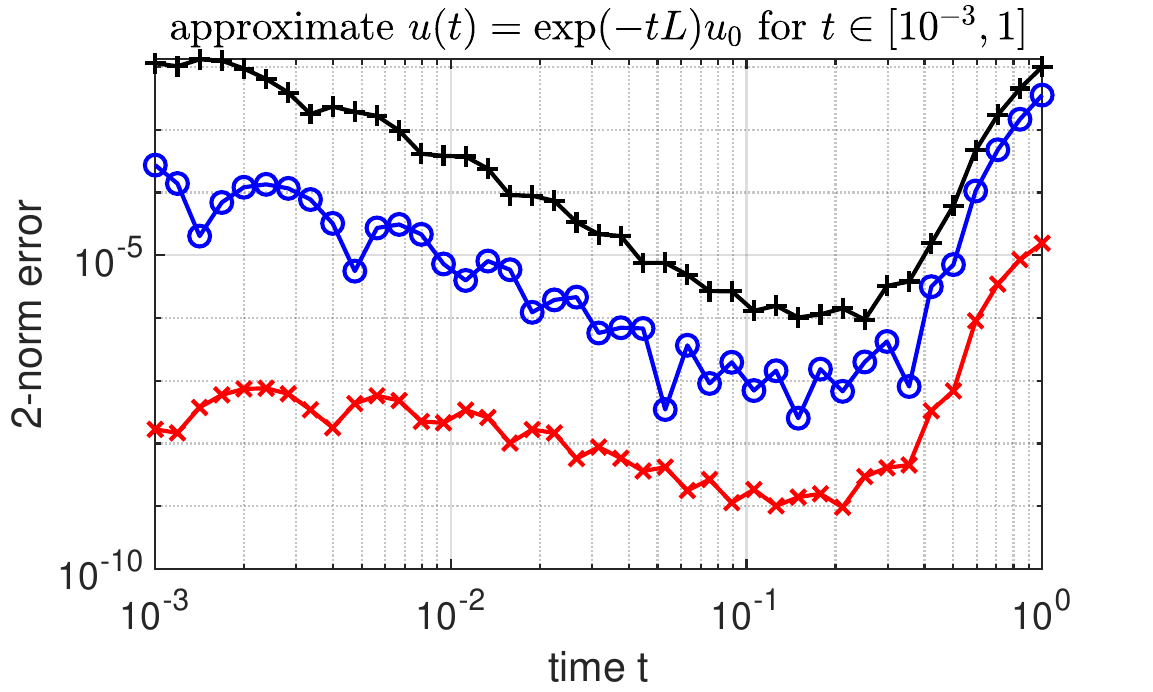}
        \caption{\textit{$n = 14$, $\tau = 10^3$.}}
        %\label{fig:sub2}
    \end{subfigure}
    
    \vspace{0.5cm} % Add some vertical space between rows
    
    \begin{subfigure}[b]{0.48\textwidth}
        \includegraphics[width=1.1\textwidth]{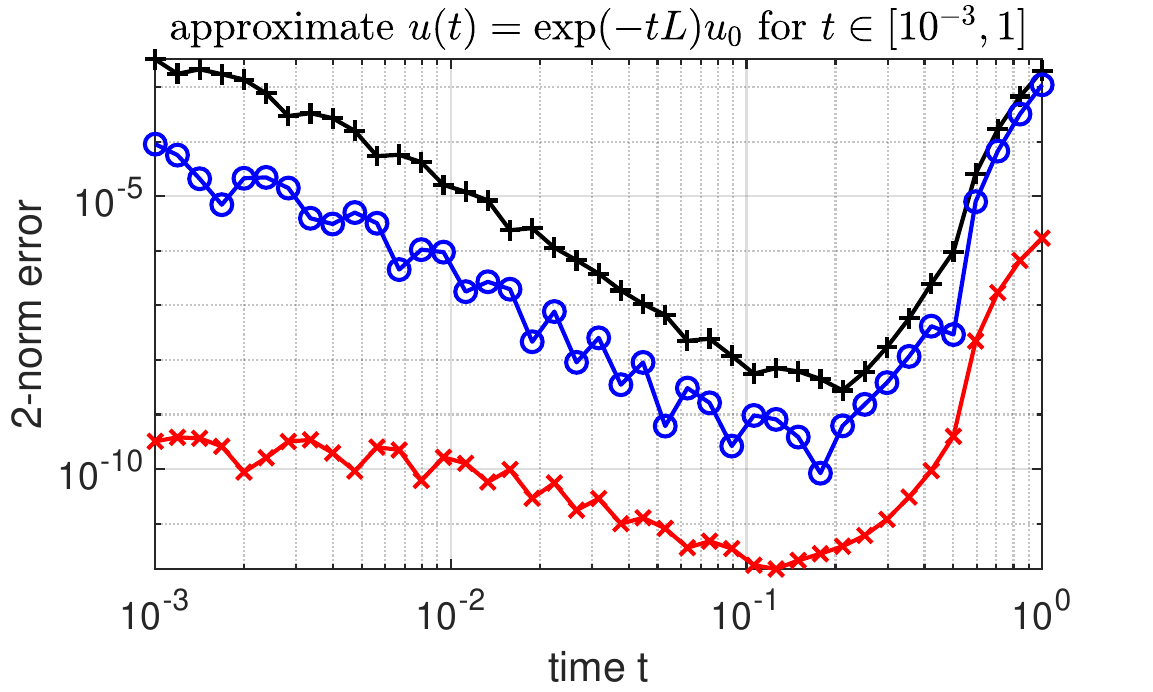}
        \caption{\textit{$n = 20$, $\tau = 10^3$.}}
        %\label{fig:sub3}
    \end{subfigure}
    \hfill
    \begin{subfigure}[b]{0.48\textwidth}
        \includegraphics[width=1.1\textwidth]{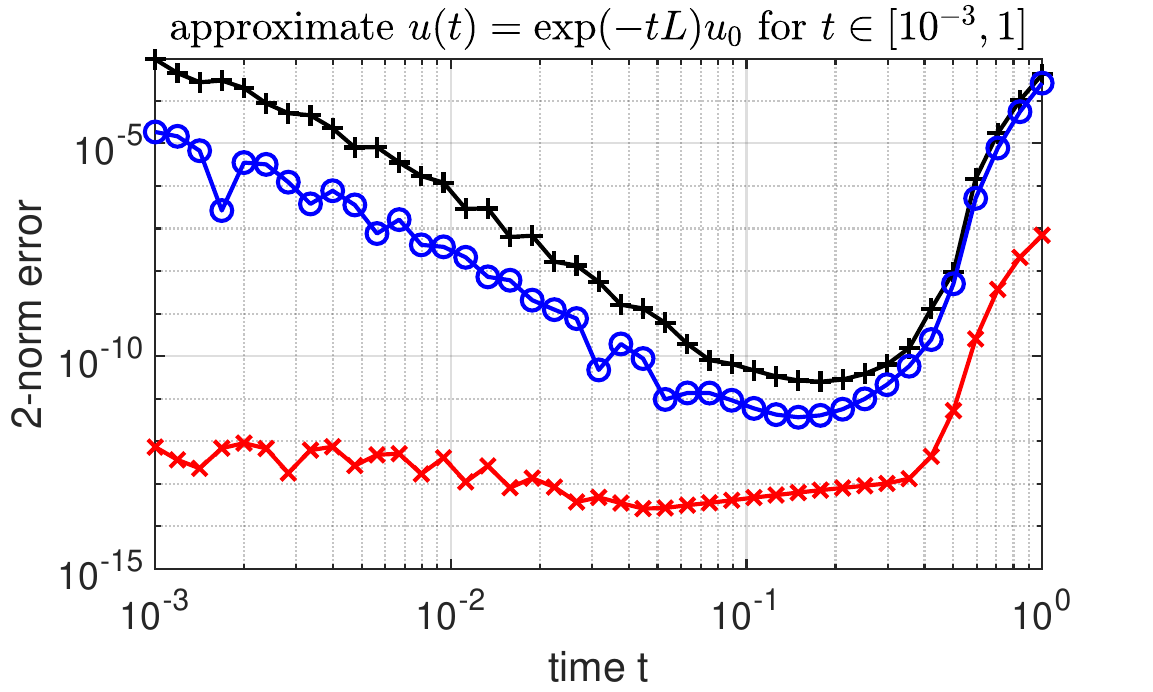}
        \caption{\textit{$n = 26$, $\tau = 10^3$.}}
        %\label{fig:sub4}
    \end{subfigure}
    
    \caption{\textit{Plot of $\Vert \widehat{\bm{u}}(t) - \bm{u}(t)\Vert_2$ using rational Chebyshev interpolation (blue), shift-and-invert Arnoldi (red), and their approximate error bound (black), for the time points $T = \texttt{logspace(-3,0,41)}$ and degrees $n = 8, 14, 20, 26$.}}
    \label{100_2by2}
\end{figure}

\section{Conclusions}\label{discussions}
We have proposed an asymptotically optimal choice of shared concentrated real poles for the rational approximation of a family of time-dependent exponential functions. The result is a generalization of~\cite{andersson1981approximation} to any time interval. 

In future work, we plan to further explore the possibilities of finding an optimal set of non-concentrated real poles supported on the negative real axis. We also intend to investigate extensions to the shared-pole approximation of functions such as the $\varphi_j$-functions, which are closely related to the exponential function and often appear in the context of exponential integrators~\cite{hochbruck2010exponential}. More work could be devoted to a detailed comparison of the performance of the different shared-pole rational approximants (with free complex poles, real poles, or concentrated real poles) in the context of large-scale computations as they arise, e.g., in Geophysics forward modeling and in inverse problems.

\section*{Acknowledgments}
We are grateful to Yoshihito Kazashi, Nick Trefethen, and Marcus Webb for providing  useful suggestions.
S.\,G. acknowledges funding from the UK's Engineering and Physical Sciences Research Council (EPSRC grant EP/Z533786/1) and the Royal Society (RS Industry Fellowship IF/R1/231032). S.\,S. acknowledges a  Dean's Doctoral Scholarship from the University of Manchester.

\bibliography{ref}
\bibliographystyle{abbrv}

\end{document}